\documentclass[12pt]{amsart}
\usepackage[colorlinks=true,citecolor=magenta,linkcolor=blue,urlcolor=magenta]{hyperref}
\usepackage{amsmath}
\usepackage[english,  activeacute]{babel}
\usepackage[utf8]{inputenc}
\usepackage{amssymb}
\usepackage{amsthm}
\usepackage{graphics,graphicx}
\usepackage{enumerate}
\usepackage{array}
\usepackage{bm}
\usepackage{a4wide}
\usepackage{color, url}
\usepackage{float}
\usepackage{tikz}
\usepackage[cm]{sfmath}

\usepackage[colorinlistoftodos]{todonotes}
\definecolor{munsell}{rgb}{0.0, 0.5, 0.69}

\theoremstyle{plain}
\newtheorem{theorem}{Theorem}[section]

\theoremstyle{definition}

\newcommand{\B}{{\mathcal B}}

\newcommand{\area}{{\texttt{area}}}
\newcommand{\sper}{{\texttt{sper}}}
\newcommand{\inn}{{\texttt{inn}}}

\title{Generating Trees and Fibonacci Polyominoes}

\date{\today}

\begin{document}

\author[J. F. Pulido]{Juan F. Pulido}
\address{Departamento de Matem\'aticas,  Universidad Nacional de Colombia,  Bogot\'a, Colombia}
\email{jupulidom@unal.edu.co}

\author[J. L. Ram\'{\i}rez]{Jos\'e L. Ram\'{\i}rez}
\address{Departamento de Matem\'aticas,  Universidad Nacional de Colombia,  Bogot\'a, Colombia}
\email{jlramirezr@unal.edu.co}

\author[A R. Vindas-Meléndez]{Andrés R. Vindas-Meléndez}
\address{Department of Mathematics, Harvey Mudd College, Claremont, CA, USA}
\email{arvm@hmc.edu}

\begin{abstract}
We study a new class of polyominoes, called  $p$-Fibonacci polyominoes, defined using $p$-Fibonacci words. 
We enumerate these polyominoes by applying  generating functions to capture geometric parameters such as area, semi-perimeter, and the number of inner points.  Additionally, we establish bijections between Fibonacci polyominoes, binary  Fibonacci words, and integer compositions with certain restrictions.
\end{abstract}

\maketitle

\section{Introduction}\label{intro}
The \emph{Fibonacci numbers} $F_n$  are among the most famous sequences in combinatorics and number theory. 
This sequence is defined recursively as \[F_{n} = F_{n-1} + F_{n-2} \text{ for all } n \geq 2,\] with initial values $F_0 = 0$ and $F_1 = 1$.  
A natural way to generalize this sequence is to consider a recursion where each term is the sum of the previous $p$ terms.   
For a positive integer $p$, the \emph{$p$-generalized Fibonacci numbers}, denoted by $F_{p,n}$, are  defined by the recursion:
    \[ F_{p,n}=    
        \sum_{i=1}^{p}F_{p,n-i}, \quad  n>1,
     \]
with the initial values $F_{p,n}=0$ for  $n=-p+2,-p+1,\dots, 0$, and $F_{p,n}=1$ for $n=1$. 
The inclusion of negative indices is meant to simplify the computation of $F_{p,n}$ for $n=2,3,\dots,p$. 

For $1<n\leq p$, the term $F_{p,n}$ is given by $2^{n-2}$. 
In particular, when $p=2$, we recover the classical Fibonacci sequence.
The $p$-generalized Fibonacci numbers have several combinatorial interpretations. 
For example, the  binary words of length $n$ that avoid  $p$ occurrences of the symbol $1$ are enumerated by   $F_{p,n+2}$ (cf. \cite{Bernini}). 
Another object enumerated by the $p$-generalized   Fibonacci sequence can be derived from generating trees.

A \emph{generating tree} is a planar, rooted, and labelled tree equipped with a \emph{production system} of labels
\[ \Omega = \begin{cases}
    (s_0) \\
    (k) \leadsto \left(e_1(k)\right), \left(e_2(k)\right), \ldots, \left(e_k(k)\right).
\end{cases} \]
The generating tree induced by the production system $\Omega$ consists of a root labeled with $s_0$, and each node labeled with $k$ generates $k$ new nodes labeled by $e_1(k),\dots, e_k(k)$.  
For example, consider the production system $\Omega_3$, given by 
\[\Omega_3=\begin{cases}(2_2)&  \\ (2_2)\leadsto (2_2)(2_1)  \\
(2_1)\leadsto (2_2)(1)\\ 
(1)\leadsto (2_2). \end{cases}\]
The generating tree induced by $\Omega_3$ has the vertex $2_2$ as its root.
If a vertex is labeled $2_2$, it generates two new nodes with labeled by $2_2$ and $2_1$. 
If a vertex is labeled $2_1$, it generates two new nodes  labeled by $2_2$ and $1$. 
Finally, if a vertex is labeled $1$, it generated one new node labeled by $2_2$.  
Figure \ref{Fibonaccistrings} illustrates the first few levels of this generating tree. 
\begin{figure}[ht!]
\centering
\includegraphics{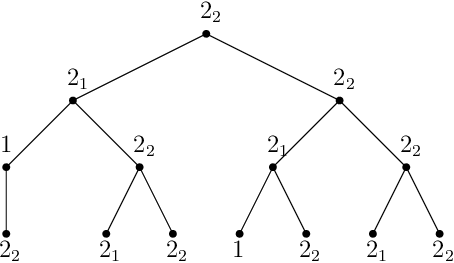}
    \caption{Generating tree induced by $\Omega_3$.}
    \label{Fibonaccistrings}
\end{figure}
 
 A production system $\Omega$ induces a sequence of positive integers $(a_n)_{n\geq 0}$, where $a_n$ is the number of nodes at level  $n$ in the generating tree.
 For our example, the first few values of the sequence $a_n$ are $1, 2, 4, 7, 13, \dots$. 
 In general, this sequence corresponds to the 3-generalized Fibonacci numbers, also known as the tribonacci numbers (sequence A000073 in the OEIS \cite{OEIS}). 
 The systematic study of generating trees and their applications was initiated by Barcucci et al.\  \cite{Barcucci}.

The production system $\Omega_3$ can be generalized as follows:
       \begin{equation}
    \Omega_p =\begin{cases}
    (2_{p-1})\\
    \left(2_{p-1}\right) \leadsto  \left(2_{p-1}\right)\left(2_{p-2}\right) \\
    \left(2_{p-2}\right) \leadsto  \left(2_{p-1}\right)\left(2_{p-3}\right) \\
    \quad \vdots \\
    \left(2_1\right) \leadsto  \left(2_{p-1}\right)\left(1\right) \\
    \left(1\right) \leadsto  \left(2_{p-1}\right).
    \label{omega_p}
\end{cases} 
\end{equation}
Baril and Do \cite{Baril} proved that the sequence associated with $\Omega_p$ encodes the $p$-generalized Fibonacci sequence $F_{p,n}$. 
The production system $\Omega_p$ allows  for the  construction of a new family of $p$-ary words, which are enumerated by the $p$-generalized Fibonacci numbers.
A $p$-ary word $u=u_1u_2\cdots u_n$ is called a \emph{$p$-Fibonacci word} of length $n$  if $u_1=p$ and for $1\leq i \leq n-1$: 
\begin{equation}
u_{i+1} = 
\begin{cases}
p, & \text{if } u_i = 1, \\
u_i-1 \text{ or } p, & \text{if } 2 \leq u_i \leq p.
\end{cases}
\label{constructionrules}
\end{equation}

We let $|w|$ denote the length of a word $w$, that is, the number of symbols of $w$.
We denote by $\epsilon$ the \emph{empty word}, that is the unique word of length zero.
For $n\geq 0$, let $\mathcal{W}_n^{(p)}$ denote the set of $p$-Fibonacci words of length $n$, and define   $\mathcal{W}^{(p)}:=\bigcup_{n\geq0}\mathcal{W}_n^{(p)}$. 
For example,
\begin{align*}
\mathcal{W}_4^{(3)}=\left\{3213,\, 3232,\, 3233,\, 3321,\, 3323,\, 3323,\, 3332,\, 3333\right\}.
\end{align*}
Note that the generating tree induced by $\Omega_p$ can be encoded by the $p$-Fibonacci words. 
Specifically, the words of length $n$ correspond to the nodes at level $n$ of the tree.
For example, Figure \ref{generatingtree3wors} illustrates the construction of the $3$-Fibonacci words.

\begin{figure}[ht!]
\centering
\includegraphics{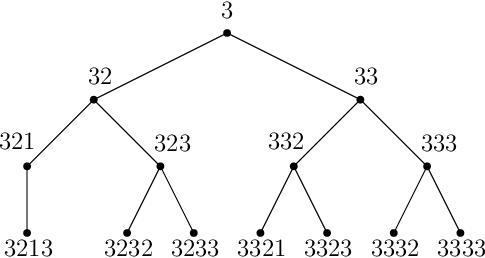}
    \caption{Initial levels of the planar tree associated to the words of $\mathcal{W}^{(3)}.$}
    \label{generatingtree3wors}
\end{figure}

Recently, the combinatorial study of different families of words  has been explored through their interpretation as polyominoes. 
A \emph{polyomino} is a finite collection of unit squares connected along their edges.
The connection between these objects is important because, although polyominoes are simple to describe, their combinatorics can be quite challenging. 
For instance, there is no known closed formula for the total number of polyominoes with a fixed area \cite{Book1}. 
This complexity has led to the study of subfamilies of polyominoes with more structure, which, in some cases, simplifies  their enumeration. 
Examples of such work include studies on $k$-ary words \cite{BBKT2, ManSha}, compositions \cite{MansourB}, Fibonacci words \cite{KirRam},  Catalan and Motzkin words \cite{Baril2, CallManRam}, permutations \cite{BBKT3},  and  others.

Research on polyominoes has primarily focused on geometric parameters, such as area, perimeter, and the number of lattice points. 
Each $p$-generalized Fibonacci word $u=u_1u_2\cdots u_n$ induces a polyomino (also called \emph{bargraph}) with $n$ columns, where the value of the symbol $u_i$ represents the number of cells in the $i$-th column. 
This new family of polyominoes is called \emph{$p$-Fibonacci polyominoes}. 
By construction, the $p$-Fibonacci polyominoes are enumerated by the $p$-generalized Fibonacci numbers $F_{p,n+1}$.
Figure \ref{poliminoes4} illustrates the $3$-Fibonacci polyominoes of length 4.

\begin{figure}[ht!]
\centering
\includegraphics{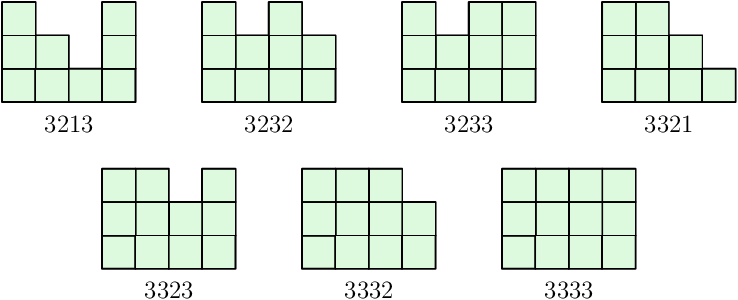}
    \caption{The $3$-Fibonacci polyominoes of length 4.}
    \label{poliminoes4}
\end{figure}

The main purpose of this paper is to enumerate the $p$-Fibonacci polyominoes  based on their length (number of columns) and geometric parameters, including area, semi-perimeter, and the number of inner points. 
We make use of generating functions to describe our results. 
We also introduce a novel method to calculate the generalized $p$-Fibonacci numbers using the three aforementioned parameters on polyominoes, namely: area, semi-perimeter, and number of inner points.

\section{Enumerating Area, Perimeter, and Lattice Points}

The area, perimeter, and number of lattice points are three fundamental parameters in the study of polyominoes. Let $P$ be a $p$-Fibonacci polyomino. 
The statistic $\area(P)$  represents the number cells comprising $P$. 
The statistic $\sper(P)$  denotes half of the perimeter of $P$, also known as the semiperimeter. 
Observe that the perimeter of a $p$-Fibonacci polyomino is always an even number.
An \emph{inner point} of $P$ is a point that is not on the perimeter of the polyomino; in other words, it is a point that belongs to exactly four cells of the polyomino.
Let $\inn(P)$ denote the number of inner points of $P$. 
For example, the $4$-Fibonacci polyomino $P$ shown in Figure \ref{parameters} satisfies $\area(P)=23$ (number of green cells), $\sper(P)=17$ (half of the perimeter), and $\inn(P)=7$ (number of blue dots).

\begin{figure}[ht]
\centering
\includegraphics{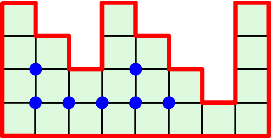}
\caption{The area, perimeter, and inner points of a polyomino.}
\label{parameters}
\end{figure}

We can interpret our polyominoes (bargraphs) as polygons. 
The area, perimeter, and number of lattice points of a polygon are connected by Pick's theorem \cite{Pick}. 
Let $T$  be a simple polygon (non-self-intersecting) with integer coordinates for its vertices. Then, the following relationship holds:
  \[
 \area(T) = \texttt{int}(T) + \frac{\texttt{ext}(T)}{2} - 1,
 \]
  where $\texttt{int}(T)$ and $\texttt{ext}(T)$ represent the number of interior and boundary lattice points of $T$, respectively.  Let $a_p(n), s_p(n)$, and $i_p(n)$ denote the total area, semi-perimeter, and number of inner points, respectively, within all  $p$-Fibonacci polyominoes with  $n$ columns.   Using Pick's theorem and summing over  all polyominoes with a fixed number of columns, we obtain the following relation:
  \[
a_p(n) = i_p(n) + s_p(n) - F_{p,n+1}.
 \]
Thus, the $p$-generalized Fibonacci numbers can be calculated by the interesting expression:
  \begin{align}\label{EqFibo}
 F_{p,n+1} = i_p(n) + s_p(n) - a_p(n), \quad n\geq 1.        \end{align}
As an example, the 4-Fibonacci number $F_{4,6}=15$ can be calculated from \eqref{EqFibo} as follows
\[F_{4,6}=15=i_4(5)+s_4(5)-a_4(5)=124+152-261.\]

In what follows, we derive the generating functions for these sequences.

\subsection{A decomposition for the Fibonacci polyominoes}

Define $\B^{(p)}_n$ as the set of $p$-Fibonacci polyominoes with $n$ columns, and let $\B^{(p,i)}_n$ denote the subset of $p$-Fibonacci polyominoes with $n$ columns where the last column contains $i$ cells. 
It follows that  \[\B^{(p)}_n=\bigcup_{i=1}^{p} \B^{(p,i)}_n.\] 
Let $\B^{(p)}$ be the set of all $p$-Fibonacci polyominoes, that is $\B^{(p)}=\bigcup_{n\geq 0} \B^{(p)}_n$. We define $\B^{(p)}_0=\{\epsilon\}$, where $\epsilon$ denotes the empty polyomino (or empty word).

To construct, or equivalently, to decompose $B_n^{(p)}$, we must determine what happens at each $B_n^{(p,i)}$ for $1\leq i \leq p$.
Let  $P$ be a polyomino in  $\B^{(p,i)}_n$. 
If $n=1$, then $P$ is a column with $p$ cells.
For $n>1$, there are two cases to consider:

\begin{enumerate}
\item[\textbf{Case 1},] $1\leq i \leq p-1$.
In this case, the polyomino  $P$ starts with a column of $p$ cells, followed by a non-empty polyomino $P'$ in $\B^{(p,i+1)}_{n-1}$. 
See Figure \ref{p1} (left) for a  representation of this decomposition.

  \item[\textbf{Case 2},] $i=p$. 
Here, the polyomino  $P$  starts with a column of $p$ cells, followed by a polyomino $P'$ in $\B^{(p,j)}_{n-1}$, for $1\leq j \leq p$. 
Refer to Figure \ref{p1} (right) for an illustration of this case.
\begin{figure}[ht!]
    \centering
  \includegraphics{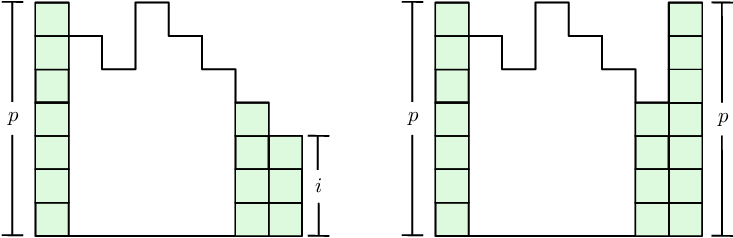}
        \caption{Decomposition of a $p$-Fibonacci polyomino.}
    \label{p1}
\end{figure}
\end{enumerate}

\subsection{Area and semi-perimeter.}

Define the multivariate generating function
\begin{equation}\label{def:multiv_gen_fun}
F_p(x;y,z):=\sum_{n\geq 0} x^n \sum_{P\in \B^{(p)}_{n}}y^{\area(P)}z^{\sper(P)}.
\end{equation}

The coefficient $x^ny^iz^j$ in the generating function  $F_p(x;y,z)$ corresponds to the number of $p$-Fibonacci polyominoes $P$ with $n$ columns such that $\area(P)=i$ and $\sper(P)=j$. 
In Theorem~\ref{teoareaper1}, we provide a rational expression for the generating function  $F_p(x;y,z)$.

\begin{theorem}\label{teoareaper1}
The generating function for the $p$-Fibonacci polyominoes with respect to
the length, semi-perimeter, and area 
is
\[F_p(x;y,z)=1+ \cfrac{\sum_{i=1}^{p}x^{p-i+1}y^{\frac12 (p + i) (p - i+1)}z^{2p-i+1}}{1 -xy^pz - \sum_{i=1}^{p-1}  x^{p-i+1}y^{\frac12 (p + i) (p-i+1)}z^{2p-2i+1} }.\]
\end{theorem}
\begin{proof}
Let $P$ be a $p$-Fibonacci polyomino in $\B_n^{(p,i)}$.  If $n=0$, then $P$ is the empty polyomino and its contribution to the generating function is the term $1$. For $n\geq 1$ we have the following cases:

\begin{enumerate}
\item[\textbf{Case 1},] $1\leq i \leq p-1$.  According to the decomposition illustrated in Figure  \ref{p1} (left), the generating function for this case is given by:
\begin{align}\label{genfunteo1}
    F_{p,i}(x;y,z)= xy^izF_{p,i+1}(x;y,z), \quad 1\leq i \leq p-1.
\end{align}

 \item[\textbf{Case 2},]  $i=p$.  If $P$ has only one column ($n=1$), then $P$ consists of a single column with $p$ cells.  Its area and semi-perimeter are  $p$ and $p+1$, respectively. Thus, its contribution to the generating function is  $xy^pz^{p+1}$. For $n>1$, the decomposition illustrated in Figure \ref{p1} (right) leads to the following  equation:
 $$xy^p\sum_{j=1}^{p} z^{p+1-j}F_{p,i}(x;y,z).$$
 \end{enumerate}
Combining these results, we obtain the following system of equations:
\[\begin{cases}
F_{p}(x;y,z)&= 1 + \sum_{i=1}^pF_{p,i}(x;y,z),\\
F_{p,i}(x;y,z)&= xy^{i}zF_{p,i+1}(x;y,z), \quad 1\leq i\leq p-1, \\
F_{p,p}(x;y,z)&= xy^{p}z^{p+1}+xy^p\sum_{j=1}^{p} z^{p+1-j}F_{p,j}(x;y,z).    
\end{cases}
\]\label{thmarea}

From \eqref{genfunteo1},  we can express each generating function $F_{p,i}(x; y, z)$ in terms of $F_{p,p}(x; y, z)$ by iterating the recurrence relation:
\begin{align*}
    F_{p,i}(x; y, z) &= (xy^i z) (xy^{i+1} z) \cdots (xy^{p-1} z) F_{p,p}(x; y, z) \\
    &= x^{p-i}y^{\frac12 (p - i) (p - 1 + i)}z^{p-i}F_{p,p}(x; y, z).
\end{align*}
Now, substitute $F_{p,i}(x;y,z)$  into the equation for $F_{p,p}(x; y, z)$:
\begin{align*}
    F_{p,p}(x; y, z) &= xy^p z^{p+1} + xy^p\sum_{j=1}^{p-1}  z^{p+1-j} F_{p,j}(x; y, z) + xy^p z F_{p,p}(x; y, z) \\
    &= xy^p z^{p+1} + xy^p\sum_{j=1}^{p-1}  x^{p-j}y^{\frac12 (p - j) (p - 1 + j)}z^{2p-2j+1}F_{p,p}(x; y, z) + xy^p z F_{p,p}(x; y, z)\\
       &= xy^p z^{p+1} + \left(\sum_{j=1}^{p-1}  x^{p-j+1}y^{p+\frac12 (p - j) (p - 1 + j)}z^{2p-2j+1} + xy^pz\right)F_{p,p}(x; y, z).
\end{align*}
Therefore, 
\[    F_{p,p}(x; y, z) = \cfrac{xy^p z^{p+1}}{1 -xy^pz - \sum_{j=1}^{p-1}  x^{p-j+1}y^{\frac12 (p + j) (p-j+1)}z^{2p-2j+1} }.
\]
Substituting this into the formula for \( F_p(x; y, z) \), we obtain
\begin{align*}
    F_p(x; y, z) &= 1 + \sum_{i=1}^{p} F_{p,i}(x; y, z) \\
    &= 1+ F_{p,p}(x; y, z)  \sum_{i=1}^{p}x^{p-i}y^{\frac12 (p - i) (p - 1 + i)}z^{p-i}\\
    & = 1+ \cfrac{\sum_{i=1}^{p}x^{p-i+1}y^{\frac12 (p + i) (p - i+1)}z^{2p-i+1}}{1 -xy^pz - \sum_{i=1}^{p-1}  x^{p-i+1}y^{\frac12 (p + i) (p-i+1)}z^{2p-2i+1} }.    
\end{align*}
\end{proof}

For example, for $p=2, 3, 4$ we obtain the generating functions: 
\begin{align*}
    F_2(x; y, z)&=\frac{1 - x y^2 z + x y^2 z^3 - x^2 y^3 z^3 + x^2 y^3 z^4}{1 - x y^2 z - x^2 y^3 z^3}\\
    &=1 + y^2 z^3 x + (y^3 z^4 + y^4 z^4) x^2 + (y^5 z^5 + y^6 z^5 + 
    y^5 z^6) x^3 + O(x^4),\\\\
    F_3(x; y, z)&=\frac{1 - x y^3 z - x^2 y^5 z^3 + x y^3 z^4 + x^2 y^5 z^5 - x^3 y^6 z^5 + 
 x^3 y^6 z^6}{1 - x y^3 z - x^2 y^5 z^3 - x^3 y^6 z^5}\\
    &=1 + y^3 z^4 x + (y^5 z^5 + y^6 z^5) x^2 + (y^6 z^6 + y^8 z^6 + 
    y^9 z^6 + y^8 z^7) x^3 + O(x^4),\\\\
    F_4(x; y, z)&=\frac{1 - x y^4 z - x^2 y^7 z^3 + x y^4 z^5 - x^3 y^9 z^5 + x^2 y^7 z^6 + 
 x^3 y^9 z^7 - x^4 y^{10} z^7 + x^4 y^{10} z^8}{1 - x y^4 z - x^2 y^7 z^3 - x^3 y^9 z^5 - x^4 y^{10} z^7}\\
    &=1 + y^4 z^5 x + (y^7 z^6 + y^8 z^6) x^2 + \bm{(y^9 z^7 + y^{11} z^7 + 
    y^{12} z^7 + y^{11} z^8) x^3} + O(x^4).    
\end{align*}

The boldface term in the generating function $F_4(x;y,z)$, above, is illustrated in Figure~\ref{grafo2}, which depicts the perimeters and areas of the  $4$-Fibonacci  polyominoes with 3 columns.

\begin{figure}[ht]
\centering
\includegraphics{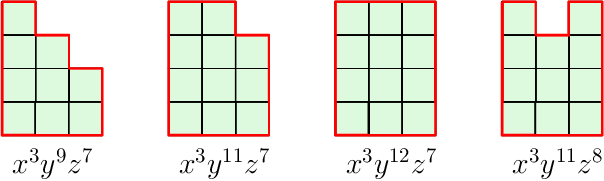}
\caption{Weights for $4$-Fibonacci  polyominoes with 3 columns.}
\label{grafo2}
\end{figure}

From Theorem~\ref{teoareaper1}, we can derive several results, including those related to the total area and semi-perimeter of $p$-Fibonacci polyominoes with a fixed number of columns.

\subsubsection{Total area of Fibonacci polyominoes}\hfill

Let $A_p(x)$ denote the generating function for the sequence $a_p(n)$,  which represents the total area of $p$-Fibonacci polyominoes with  $n$ columns.
Using the definition of  $F_p(x;y,z)$ in Equation (\ref{def:multiv_gen_fun}), we obtain:
\begin{align*}
A_p(x)&=\sum_{n\geq 0}a_p(n)x^n=\left.\frac{\partial F_p(x;y,1)}{\partial y}\right|_{y=1}=\frac{\sum_{i=1}^pi(2p - i + 1)x^i}{2(1-x-x^2-\cdots -x^p)^2}\\
&=\frac{x (p^2 (1-x)^2 x^p + 2 x (1 - x^p) - 
   p (1 - x) (2 - x^p + x^{p+1}))}{2 (-1 + x) (1 - 2 x + x^{p+1})^2}.
\end{align*}

In Table \ref{table3}, we present the first few values of the sequence 
$a_p(n)$. 
Note that the sequences corresponding to $n=2$ and $n=3$ where $p$ varies match OEIS entries $A004767$ and $A004170$, respectively.
Though for $n\geq 4$, these sequences do not appear in the OEIS.

\begin{table}[H]
\begin{center}
\begin{tabular}{|c|cccccccccc|} \hline
$p \backslash n$ & 1 & 2 & 3  & 4 & 5 & 6  & 7 & 8 & 9 &  10  \\ \hline
2 & 2& 7& 16& 35& 70& 136& 256& 473& 860& 1545  \\ \hline
 3 & 3& 11& 31& 73& 168& 370& 790& 1658& 3425& 6989  \\ \hline
 4 & 4& 15& 43& 111& 261& 602& 1350& 2966& 6414& 13714  \\ \hline
 5 & 5& 19& 55& 143& 351& 816& 1865& 4178& 9218& 20094   \\ \hline
\end{tabular}
\end{center}
\caption{Sequence $a_p(n)$ for $1\leq n \leq  10$ and $2\leq p \leq 5$.}
\label{table3}
\end{table}

\subsubsection{Number of Fibonacci polyominoes with fixed area}\label{fixed area}\hfill

Let $d_p(n)$ denote the number of $p$-Fibonacci polyominoes with an area of $n$. 
The generating function for the sequence $d_p(n)$ can be obtained from Theorem \ref{teoareaper1} by setting $z=1$ and $x=1$:
\[\sum_{n\geq 0}d_p(n)y^n=\frac{1}{1-\sum_{i=1}^p y^{\frac12 (p + i) (p-i+1)}}.\]

In Table \ref{tablecompositions}, we present the first few values of the sequence 
$d_p(n)$. 

\begin{table}[H]
\begin{center}
\begin{tabular}{|c|cccccccccc|} \hline
$p \backslash n$ & 1 & 2 & 3  & 4 & 5 & 6  & 7 & 8 & 9 &  10  \\ \hline
2 & 0& 1& 1& 1& 2& 2& 3& 4& 4& 7  \\ \hline
 3 & 0& 0& 1& 0& 1& 2& 0& 2& 3& 1  \\ \hline
 4 & 0& 0& 0& 1& 0& 0& 1& 1& 1& 1  \\ \hline
 5 & 0& 0& 0& 0& 1& 0& 0& 0& 1& 1   \\ \hline
\end{tabular}
\end{center}
\caption{Sequence $d_p(n)$ for $1\leq n\leq 10$ and $1\leq p \leq 5$.}
\label{tablecompositions}
\end{table}

Note that $d_p(n)$ coincides with the number of composition of $n$ using parts from the set
\[A_p\mathrel{:}=\left\{\frac12 (p + i) (p-i+1): 1 \leq i \leq p\right\}.\]
Recall that a \emph{composition} of a positive integer $n$ is a sequence of positive integers  $\sigma=(\sigma_1, \sigma_2,\ldots, \sigma_\ell)$ such that $\sigma_1+\sigma_2+\cdots +\sigma_\ell=n$.
The summands $\sigma_i$ are called \emph{parts} of the composition. 

\begin{theorem}\label{bijcom}
The set of $p$-generalized Fibonacci polyominoes of area $n$ are in bijection with the set of compositions of $n$ with parts in the set $A_p$.
\end{theorem}
\begin{proof}
Given a $p$-Fibonacci polyomino $P$ of area $n$, associated to the generalized Fibonacci word $u = u_1 u_2 \cdots u_k$, we construct a composition of $n$ with parts in  $A_p$ as follows. First, the word $\omega$ is read from left to right and divided into blocks, each starting with $p$ and consisting of the largest consecutive decreasing sequence of digits in $u$. The sum of the digits in each block is then recorded as a part of the composition. The resulting composition has its parts ordered according to the blocks of $u$.

For example, when \(p=3\) and \(n=11\), the word \(u = 32321\) is divided into the blocks \(32\) and \(321\). Summing the entries of these blocks yields the composition \(5 + 6\), where \(5 = 3+2\) and \(6 = 3+2+1\). Since the sum of the digits $u_1,\dots, u_k$ equals the area of the polyomino, the resulting composition is indeed a composition of $n$. Moreover, each part of the resulting composition is of the form  $$\sum_{\ell=0}^{j}(p-\ell)=\frac{1}{2}(j+1)(2p-j),$$ 
for some $0\leq j \leq p-1$. Taking $i=p-j$, we verify that $1\leq i \leq p$ and
$$\sum_{\ell=0}^{j}(p-\ell)=\frac{1}{2}(p-i+1)(p+i) \in A_p.$$

Conversely, to reconstruct a $p$-Fibonacci polyomino of area $n$  from a composition \(\sigma = \sigma_1 + \sigma_2 + \cdots + \sigma_k\), where each \(\sigma_i \in A_p\), we decompose each part \(\sigma_i\) into the sum of the largest consecutive decreasing sequence starting at \(p\). Specifically, for each \(\sigma_i\), we find the largest \(k\) such that 
\[
\sigma_i = \sum_{\ell=0}^k (p-\ell).
\]
This determines the sequence $p, p-1, \dots, p-k$. Concatenating these sequences in the order of the composition gives the desired \(p\)-generalized Fibonacci word representation of the polyomino.

For instance, given the composition \(\sigma = 6 + 5\) with \(p = 3\), the part \(6\) decomposes as \(6 = 3 + 2 + 1\), corresponding to the block \(321\), and the part \(5\) decomposes as \(5 = 3 + 2\), corresponding to the block \(32\). Concatenating these blocks yields the word \(u = 32132\).
 \end{proof}
 
\subsubsection{Total semi-perimeter of Fibonacci polyominoes}\hfill

Let $S_p(x)$ denote the generating function for the sequence $s_p(n)$,  which represents the total semi-perimeter of $p$-Fibonacci polyominoes with  $n$ columns. Therefore,
\begin{align*}
S_p(x)&=\sum_{n\geq 0}s_p(n)x^n=\left.\frac{\partial F_p(x;1,z)}{\partial z}\right|_{z=1}\\
&=\frac{p (1-x) x (1 - 2 x - 2 x^p + 3 x^{p+1}) - 
 x (1 - x^p) (-1 + x - x^2 + x^{p+2})}{(1 - x) (1 - 2 x + x^{p+1})^2}.
\end{align*}
In Table \ref{table3b} we show the first few values of the sequence $s_p(n)$. 
These sequences do not appear in the OEIS for $n\geq 3$.  
\begin{table}[H]
\begin{center}
\begin{tabular}{|c|cccccccccc|} \hline
$p \backslash n$ & 1 & 2 & 3  & 4 & 5 & 6  & 7 & 8 & 9 &  10  \\ \hline
2 & 3& 8& 16& 33& 63& 119& 219& 398& 714& 1269  \\ \hline
 3 & 4& 10& 25& 54& 118& 251& 521& 1071& 2176& 4380  \\ \hline
 4 & 5& 12& 29& 69& 152& 335& 727& 1557& 3297& 6931  \\ \hline
 5 & 6& 14& 33& 77& 177& 390& 856& 1859& 4001& 8545   \\ \hline
\end{tabular}
\end{center}
\caption{Sequence $s_p(n)$ for $1\leq n \leq  10$ and $2\leq p \leq 5$.}
\label{table3b}
\end{table}

\subsection{Inner points in the Fibonacci polyominoes} 

As with the area and semi-perimeter statistics, we define generating functions for the study of inner points as follows:
\[G_p(x;q):=\sum_{n\geq 0} x^n \sum_{P\in \B_{n}^{(p)}}q^{\inn(P)} \quad \text{and} \quad
G_{p,i}(x;q):=\sum_{n\geq 0} x^n \sum_{P\in \B_{n}^{(p,i)}}q^{\inn(P)}.\]

\begin{theorem}\label{teoareaper2}
The generating function for the $p$-Fibonacci polyominoes with respect to
the length and inner points 
is  given by
\[
    G_{p}(x;q)= 1 + \frac{\sum_{i=1}^{p} x^{p-i+1} q^{\frac 12 (p-i)(p+i-3)}}{1 - \sum_{i=1}^{p} x^{p-i+1} q^{\frac 12((p - i) (p + i - 3) + 2 (i - 1))}}.
    \]
    \end{theorem}

\begin{proof}
Let $P$ be a $p$-Fibonacci polyomino with $n$ columns.  If $n=0$, then $P$ is the empty polyomino, and its contribution to the generating function is the term $1$. For $n\geq 1$, we can use a similar  argument as given in the proof of Theorem \ref{teoareaper1}. In this case, we  have the system of equations:
\[\begin{cases}
G_{p}(x;q)&= 1 + \sum_{i=1}^pG_{p,i}(x;q),\\
G_{p,i}(x;q)&= xq^{i-1}G_{p,i+1}(x;q), \quad 1\leq i\leq p-1, \\
G_{p,p}(x;q)&= x+\sum_{i=1}^{p}xq^{i-1}G_{p,i}(x;q).    
\end{cases}
\]
Solving this system of equations, we obtain the desired result. 
\end{proof}

For example, for $p=2, 3, 4$ we obtain the following generating functions: 
\begin{align*}
    G_2(x; q)&=\frac{1 + x(1 - q)}{1 - x(q+x)}\\
    &=1 + x + (1 + q) x^2 + (1 + q + q^2) x^3 + (1 + 2 q + q^2 + q^3) x^4 + O(x^5),\\\\
    G_3(x; q)&=\frac{1 + (1 - q^2) x + (1 - q) q x^2}{1 - q x^3 - q^2 x (1 + x)}\\
    &=1 + x + (q + q^2) x^2 + (q + q^2 + q^3 + q^4) x^3 + (q + 2 q^3 + 
    2 q^4 + q^5 + q^6) x^4 + O(x^5),\\\\
G_4(x; q)&=\frac{1 + (1 - q^3) x + q^2 (1 - q^2) x^2 + (1 - q) q^3 x^3}{1 - q^4 x^2 (1 + x) - q^3 (x + x^4)}\\
    &=1 + x + (q^2 + q^3) x^2 + \bm{(q^3 + q^4 + q^5 + q^6) x^3} + O(x^4).    
\end{align*}

The boldface term in the previous generating function, $G_4(x;q)$, is illustrated in Figure~\ref{grafo3}, which depicts the inner points (blue dots) of the  $4$-Fibonacci  polyominoes with 3 columns.

\begin{figure}[ht]
\centering
\includegraphics{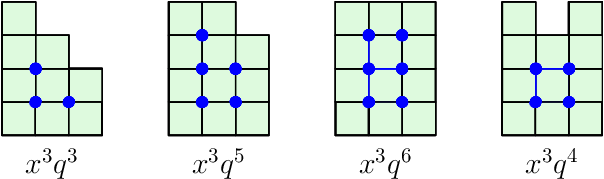}
\caption{Weights for $4$-Fibonacci  polyominoes with 3 columns.}
\label{grafo3}
\end{figure}

\newpage

\subsubsection{Total number of inner points of Fibonaci polyominoes} \hfill

Let $I_p(x)$ denote the generating function for the sequence $i_p(n)$,  which represents the total number of inner points of the $p$-Fibonacci polyominoes with  $n$ columns. Therefore,
\begin{align*}
I_p(x)&=\sum_{n\geq 0}i_p(n)x^n=\left.\frac{\partial G_p(x;q)}{\partial q}\right|_{q=1}\\
&=\frac{x ((6 - 4 p) x - (2 - 4 p) x^2 - (3 - p)p x^p - 
   2 (2- p)^2 x^{p+1} + (2 - 5 p + p^2)x^{p+2} + 2x^{2p+1})}{2 (-1 + x) (1 - 2 x + x^{p+1})^2}.
\end{align*}

In Table \ref{table4} we show the first few values of the sequence $i_p(n)$. These sequence do not appear in the OEIS.  
\begin{table}[H]
\begin{center}
\begin{tabular}{|c|cccccccccc|} \hline
$p \backslash n$ & 1 & 2 & 3  & 4 & 5 & 6  & 7 & 8 & 9 &  10  \\ \hline
2 & 0& 1& 3& 7& 15& 30& 58& 109& 201& 365  \\ \hline
 3 & 0& 3& 10& 26& 63& 143& 313& 668& 1398& 2883  \\ \hline
 4 & 0& 5& 18& 50& 124& 296& 679& 1517& 3325& 7184 \\ \hline
 5 & 0& 7& 26& 74& 190& 457& 1070& 2439& 5453& 12013  \\ \hline
\end{tabular}
\end{center}
\caption{Sequence $i_p(n)$, for $1\leq n \leq  10$ and $2\leq p \leq 5$.}
\label{table4}
\end{table}

\subsection{A Bijection}

\begin{theorem}\label{thm:binary}
The $p$-generalized Fibonacci words of length $n$ are in bijection with the binary words avoiding $p$ consecutive $1$'s of length $n-1$.
\end{theorem}
\begin{proof}
For a given generalized Fibonacci word $u=u_1u_2\cdots u_n$, we construct a binary word of length $n-1$ by examining the increases and decreases in $u$. Reading the word from left to right, we assign $0$ for increases or unchanged digits and $1$ for decreases. For example, the word $32323$ corresponds to the binary word $1010$. This binary word avoids $p$ consecutive occurrences of the digit 1, as each word has at most $p-1$ ascents. Furthermore, the word obtained is unique because the words are fully determined by the increases and decreases.
 
Let $w$ be a binary word of length $n-1$ that avoids $p$ consecutive occurrences of the digit $1$. Since each generalized Fibonacci word starts with an occurrence of the symbol $p$,  we construct the generalized Fibonacci word by appending the allowed descending digit at the right when there is a $0$ and the allowed ascending or unchanged digit when there is a 1. For instance, the word $1011011$ that avoids 3 consecutive occurrences of 1 generates the $3$-generalized Fibonacci word of length 8, $u=33233233$.  By construction the word generated is a generalized Fibonacci word and is unique since it is fully determined by the increases and decreases once we set the first digit of $u$ to be the digit $p$.
\end{proof}

Now that we have established a bijection between $p$-generalized Fibonacci words of length $n$ and binary words avoiding $p$ consecutive $1$'s of length $n-1$. We conclude by observing that the subset of $p$-generalized Fibonacci words of length $n$ or equivalently the $p$-generalized Fibonacci polyominoes with area $n$ (discussed in Subsubsection \ref{fixed area})  are also in bijection with a subset of generating trees (as introduced in Section \ref{intro} and seen in Figure \ref{generatingtree3wors}), compositions (Theorem \ref{bijcom}), and binary words (following Theorem \ref{thm:binary})  as exemplified in the following figure. 

\begin{figure}[ht!]
\centering
\includegraphics[scale=0.8]{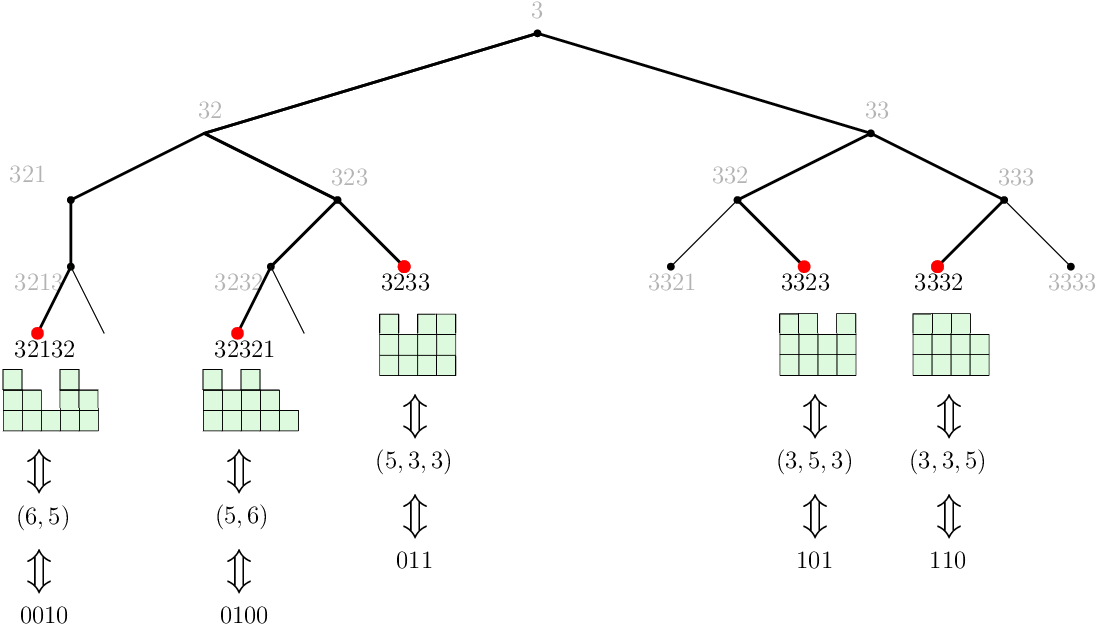}
\caption{Family of Bijections.}
\label{Bij}
\end{figure}

\section*{Acknowledgements}
The first two authors were partially supported by Universidad Nacional de Colombia, Project No. 57340.


\bibliographystyle{amsplain}
\bibliography{bibliography}


\end{document}